\documentclass{amsart}

\usepackage{latexsym,enumerate}
\usepackage{amsmath,amsthm,amsopn,amstext,amscd,amsfonts,amssymb}
\usepackage[ansinew]{inputenc}
\usepackage{verbatim}
\usepackage{graphicx}
\usepackage{pstricks}
\usepackage{wasysym}
\usepackage[all]{xy}
\usepackage{epsfig} 
\usepackage{graphicx,psfrag} 
\usepackage{subfigure}
\usepackage{pstricks,pst-node}

\usepackage{multicol}
\usepackage{color}
\usepackage{colortbl}

\newcommand{\RR}{\mathbb{R}}
\newcommand{\ZZ}{\mathbb{Z}}

\newtheorem{theorem}{Theorem}[section]
\newtheorem{lemma}[theorem]{Lemma}
\newtheorem{proposition}[theorem]{Proposition}
\newtheorem{corollary}[theorem]{Corollary}

\newtheorem{example}[theorem]{Example}
\newtheorem{remark}[theorem]{Remark}

\newcommand{\spb}[1]{\smallskip}
\newcommand{\mpb}[1]{\medskip}
\newcommand{\bpb}[1]{\bigskip}

\renewcommand{\l}{\lambda}

\begin{document}

\DeclareGraphicsExtensions{.jpg,.pdf,.mps,.png}

\title{Upper and Lower Bounds for Generalized Wiener Indices on unicyclic graphs}

\author[\'{A}lvaro Mart\'{\i}nez-P\'erez]{\'{A}lvaro Mart\'{\i}nez-P\'erez$^{(1)}$}
\address{ Facultad CC. Sociales de Talavera,
Avda. Real F\'abrica de Seda, s/n. 45600 Talavera de la Reina, Toledo, Spain}
\email{alvaro.martinezperez@uclm.es}
\thanks{$^{(1)}$ Supported in part by a grant
from Ministerio de Ciencia, Innovaci\'on y Universidades (PGC2018-098321-B-I00), Spain.
}

\author[Jos\'e M. Rodr{\'\i}guez]{Jos\'e M. Rodr{\'\i}guez$^{(2)}$}
\address{Departamento de Matem\'aticas, Universidad Carlos III de Madrid,
Avenida de la Universidad 30, 28911 Legan\'es, Madrid, Spain}
\email{jomaro@math.uc3m.es}
\thanks{$^{(2)}$ Supported by a grant from Agencia Estatal de Investigaci\'on (PID2019-106433GB-I00 / AEI / 10.13039/501100011033), Spain,
and by the Madrid Government (Comunidad de Madrid-Spain) under the Multiannual Agreement with UC3M in the line of Excellence of University Professors (EPUC3M23), and in the context of the V PRICIT (Regional Programme of Research and Technological Innovation).
}

\date{\today}

\begin{abstract}
The aim of this paper is to obtain new inequalities for a large family of generalizations of the Wiener Index
and to characterize the set of extremal graphs with respect to them. Our main results provide upper and lower bounds for these topological indices on unicyclic graphs.
\end{abstract}

\maketitle{}


%

\section{Introduction}

A topological descriptor is a single number that it is computed on the molecular graph of a compound and represents some chemical structure in terms of this graph. These descriptors are common and relevant in the field of mathematical chemistry and especially in the QSPR/QSAR investigations.

When a topological descriptor correlates with a molecular property of certain chemical compounds, then it is called a topological index. Thus, topological indices capture some essential physicochemical property into a single number and this can be used to analyze those properties. This can be very interesting for practitioners. For example, although only about 1000 benzenoid hydrocarbons are known, the number of theoretically
possible benzenoid hydrocarbons is huge. If we consider, for instance, the number of
possible benzenoid hydrocarbons with 35 benzene rings, it is $5.85\times 10^{21}$ \cite{NGJ}.
Therefore, a good model capable of predicting physico-chemical properties  of currently unknown species is extremely useful.

The main reason for the use of topological indices is to obtain predictions of some property of certain molecules (see, e.g., \cite{Estrada3}, \cite{Gutman7}, \cite{Gutman8}, \cite{RPL}).
Whenever a topological descriptor shows a better correlation for some property on some group of chemical compounds a new topological index appears. This way hundreds of topological indices have been defined and studied, starting with the seminal work by Wiener, \cite{Wi}, who found a correlation between his index and  paraffin boiling points.

The \emph{Wiener index} of $G$ is defined as
$$
W(G)=\sum_{\{ u,v\}\subseteq V(G)} d(u,v),
$$
where $\{ u,v\}$ runs over every pair of vertices in $G$.

Following this work, several versions of this index have been appearing since then, each one better adapted for its purpose. Let us mention, for example, the hyper-Wiener index, see \cite{KLG,R4} or the Harari index \cite{IBB,PNTM}.


A natural problem in the study of topological indices is, given some fixed parameters, to find the graphs that minimize (or maximize) their value on a certain set of graphs satisfying the restrictions given by the parameters (see, e.g., \cite{BE1}, \cite{BE2}, \cite{Cruz}, \cite{Das4}, \cite{Du2}, \cite{Du3}, \cite{Edwards}, \cite{Gutman32}).

Herein, instead of doing this individually on each modified version of the Wiener index we consider a natural generalization of it which encompasses those mentioned above and try to work as generally as possible. In previous works we used the same strategy with trees, see \cite{MR2}. Now, we are considering unicyclic graphs. A \emph{unicyclic} graph is a graph containing exactly one cycle \cite[p.41]{Harary}.
It is well known that if $G$ is a unicyclic graph with $n$ vertices, then $G$ has $n$ edges.

The aim of this paper is to obtain new inequalities for a large family of topological indices restricted to unicyclic graphs, fixing or not the number of pendant vertices, and to characterize the extremal unicyclic graphs with respect to them. This problem, for other type of indices, was also addressed in \cite{MR3}.

Throughout this work, $G=(V (G),E (G))$ denotes a (non-oriented) finite connected simple (without multiple edges and loops) non-trivial ($E(G) \neq \emptyset$) graph.
Note that the connectivity of $G$ is not an important restriction, since every molecular graph is connected.

\section{Wiener index and its generalizations}

Motivated by the Wiener index, Randi\'c introduced in \cite{R4} an extension of the Wiener index for trees, and this has come
to be known as the \emph{hyper-Wiener index}. In \cite{KLG}, this extension was generalized to graphs as
\[
WW(G) = \frac12 \sum_{\{ u,v\}\subseteq V(G)} d(u,v) + \frac12 \sum_{\{ u,v\}\subseteq V(G)} d(u,v)^2.
\]
$WW(G)$ has been useful in correlations (see, e.g., \cite{Gutman} and the references therein).
For information about the hyper-Wiener index in mathematics see, e.g., \cite{Cash}, \cite{Gutman}, \cite{KZG}.

Also, it is interesting to generalize the Wiener index in the following way
\[
W^{\lambda}(G) = \sum_{\{ u,v\}\subseteq V(G)} d(u,v)^\l,
\]
with $\l \in \RR$. Obviously, if $\l=1$, then $W^\l$
coincides with the ordinary Wiener index $W$.
Note that $W^{-2}$ is the Harary index; $W^{-1}$ is the reciprocal Wiener index; the quantity $W^{2}$ is closely related to the hyper-Wiener index,
since $WW = (W^1 + W^2)/2$. Another topological index, proposed in \cite{TSZ} is expressed in terms of $W^1$, $W^2$ and $W^3$ as $(2W^1+3W^2+W^3)/6$.
See \cite{KG} for more connections of the same kind.

Three different variants of the $q$-Wiener index ($q>0, \,q \neq 1$) were defined in \cite{ZGLM} as
$$
\begin{aligned}
W_1(G,q) & = \sum_{\{ u,v\}\subseteq V(G)} [d(u,v)]_q,
\\
W_2(G,q) & = \sum_{\{ u,v\}\subseteq V(G)} [d(u,v)]_q \, q^{L-d(u,v)},
\\
W_3(G,q) & = \sum_{\{ u,v\}\subseteq V(G)} [d(u,v)]_q \, q^{d(u,v)},
\end{aligned}
$$
where $L$ is the diameter of $G$, and
$$
[k]_{q} = \frac{1-q^k}{1-q} = 1+q+q^2+ \cdots + q^{k-1}.
$$
Since $\lim_{q\to 1}[k]_{q}=k$, we have
$$
\lim_{q\to 1} W_1(G,q)
= \lim_{q\to 1} W_2(G,q)
= \lim_{q\to 1} W_3(G,q)
= W(G).
$$

Given any function $h:\ZZ^+ \rightarrow \RR$, the $h$-\emph{Wiener index} of $G$ is defined as
\[
W_{h}(G)
=\sum_{\{ u,v\}\subseteq V(G)} h\big(d(u,v)\big),
\]
This general approach allows to study in a unified way the previous indices.

\medskip

If $P_n$ is the path graph with $n$ vertices, then
\[W_h(P_n)=\sum_{1\leq i < j\leq n}h(j-i)=\sum_{k=1}^{n-1}(n-k)h(k).\]

If $C_n$ is the cycle graph with $n$ vertices, then
\begin{itemize}
	\item if $n$ is odd,
\[W_h(C_n)=\sum_{j=1}^{\frac{n-1}{2}}n h(j),\]
	\item if $n$ is even,
\[W_h(C_n)=\sum_{j=1}^{\frac{n}{2}-1} nh(j)+\frac{n}{2} h\Big(\frac{n}{2}\Big).\]
\end{itemize}

\medskip

Given $n\geq 4$, let $J_n$ be the graph obtained by identifying a vertex from a cycle $C_3$ and the vertex with degree $n-3$ of a star graph with $n-2$ vertices, $S_{n-2}$. Then,
\[W_h(J_n)=n h(1)+\frac{1}{2} n(n-3) h(2).\]

Given $3\leq r \leq n$ and a function $h:\ZZ^+ \rightarrow \RR$, let us define the function $F_h(r,n)$ as follows:
\begin{itemize}
	\item if $r$ is odd,
\begin{equation}\label{eq:odd}
F_h(r,n):=\sum_{j=1}^{\frac{r-1}{2}}rh(j)+\sum_{j=1}^{n-r}(n-r+1-j)h(j)+2\sum_{k=1}^{n-r}\sum_{j=1}^{\frac{r-1}{2}}h(k+j),
\end{equation}
	\item if $r$ is even,
\begin{equation}\label{eq:even}
F_h(r,n):=\sum_{j=1}^{\frac{r}{2}-1}rh(j)+\frac{r}{2}h\Big(\frac{r}{2}\Big)+
\sum_{j=1}^{n-r}(n-r+1-j)h(j)+2\sum_{k=1}^{n-r}\sum_{j=1}^{\frac{r}{2}-1}h(k+j)+\sum_{k=1}^{n-r}h\Big(\frac{r}{2}+k\Big).
\end{equation}
\end{itemize}

Let $G_{r,n}$ be the graph obtained by identifying a vertex from a cycle $C_r$ and a vertex with degree 1 of a path graph $P_{n-r+1}$.
Note that $G_{n,n}=C_n$.

As usual, if $a < b$, we use the convention
$$
\sum_{j=b}^{a}A(j) = 0.
$$

\begin{figure}[ht]
\centering
\includegraphics[scale=0.4]{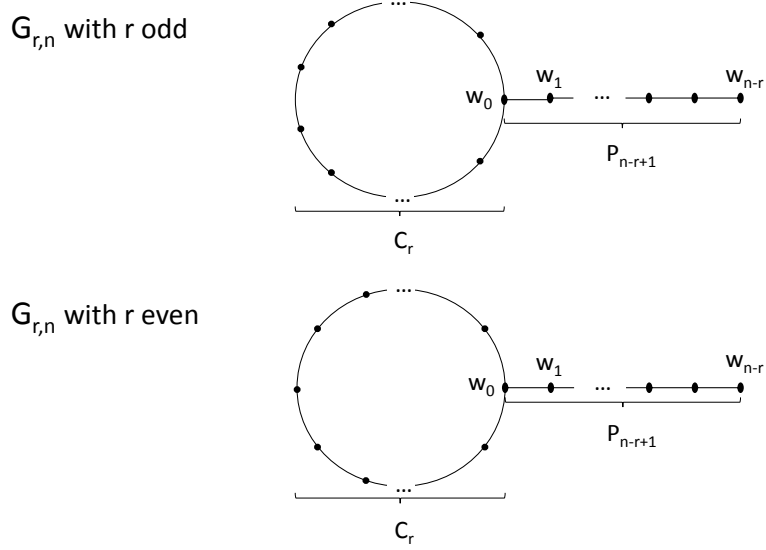}
\caption{The graph $G_{r,n}$ is obtained by identifying a vertex from a cycle $C_r$ and a vertex with degree 1 of a path graph $P_{n-r+1}$.}
\label{G_rn}
\end{figure}

\begin{proposition}\label{p:F} Given $3\leq r \leq n$ and a function $h:\ZZ^+ \rightarrow \RR$, then
\[W_h(G_{r,n})=F_h(r,n).\]
\end{proposition}

\begin{proof} Consider the graph $G_{r,n}$, let $w_0$ be the identified vertex which belongs to the cycle and the path and let $w_1,w_2,\dots,w_{n-r}$ be the vertices in the path $P_{n-r+1}$ with $d(w_0,w_k)=k$. See Figure \ref{G_rn}.

Then,
 $$W_h(G_{r,n})=W_h(C_r)+W_h(P_{n-r+1})+\sum_{w_0\neq v\in C_r, \, w_0\neq w_i\in P_{n-r+1}}  h(d(v,w_i)).$$

Suppose $r$ is odd. Then, notice that for every $1\leq k\leq n-r$,
\[\sum_{w_0\neq v\in C_r}h(d(v,w_k))
=2\sum_{j=1}^{\frac{r-1}{2}} h(j+k).\]
Thus,
$$W_h(G_{r,n})=\sum_{j=1}^{\frac{r-1}{2}}rh(j)+\sum_{j=1}^{n-r}(n-r+1-j)h(j)+2\sum_{k=1}^{n-r}\sum_{j=1}^{\frac{r-1}{2}}h(k+j).$$

In particular,
$$F_h(3,n)=3h(1)+\sum_{j=1}^{n-3}(n-2-j)h(j)+2\sum_{k=1}^{n-3}h(k+1)=
nh(1)+\sum_{j=2}^{n-2}(n-j)h(j).$$

Suppose $r$ is even. Then, notice that for every $1\leq k\leq n-r$,
\[\sum_{w_0\neq v\in C_r}h(d(v,w_k))
=2\sum_{j=1}^{\frac{r}{2}-1} h(j+k)+h\Big(\frac{r}{2}+k\Big).\]
Thus,
$$W_h(G_{r,n})=\sum_{j=1}^{\frac{r}{2}-1}rh(j)+\frac{r}{2}  h\Big(\frac{r}{2}\Big)+
\sum_{j=1}^{n-r}(n-r+1-j)h(j)+2\sum_{k=1}^{n-r}\sum_{j=1}^{\frac{r}{2}-1} h(k+j)+\sum_{k=1}^{n-r} h\Big(\frac{r}{2}+k\Big).$$
\end{proof}

\begin{remark} \label{r:pares}
Since any graph with $n$ vertices has $\frac12 n(n-1)$ pairs of vertices, we have
$$
\frac12 \,n(n-1)
= n+\sum_{j=2}^{n-2}(n-j)
=\sum_{j=1}^{\frac{r-1}{2}}r +\sum_{j=1}^{n-r}(n-r+1-j) +2\sum_{k=1}^{n-r}\sum_{j=1}^{\frac{r-1}{2}} 1
$$
if $r$ is odd, and
$$
\frac12 \,n(n-1)
= n+\sum_{j=2}^{n-2}(n-j)
=\sum_{j=1}^{\frac{r}{2}-1}r +\frac{r}{2} +
\sum_{j=1}^{n-r}(n-r+1-j) +2\sum_{k=1}^{n-r}\sum_{j=1}^{\frac{r}{2}-1} 1 +\sum_{k=1}^{n-r} 1
$$
if $r$ is even.
\end{remark}

\begin{lemma}\label{l:max_even} Given $4\leq r \leq n$ with $r$ even and a strictly increasing function $h:\ZZ^+ \rightarrow \RR$, then $F_h(3,n)>F_h(r,n)$. If $h$ is a strictly decreasing function, then $F_h(3,n)<F_h(r,n)$.
\end{lemma}

\begin{proof}
First of all, note that the second statement is a consequence of the first one, if we consider the function $-h$.
Hence, we can assume that $h$ is a strictly increasing function.
We have
$$
\begin{aligned}
& F_h(3,n)>F_h(r,n)
\\
& \Leftrightarrow \quad \sum_{j=2}^{n-r}(n-j)h(j)+\sum_{j=n-r+1}^{n-2}(n-j)h(j)
\\
& \qquad \quad > \sum_{j=2}^{\frac{r}{2}-1}rh(j)+\frac{r}{2}h\Big(\frac{r}{2}\Big)+\sum_{j=\frac{r}{2}+1}^{n-\frac{r}{2}}h(j)+\sum_{j=2}^{n-r}(n-r+1-j)h(j)+2\sum_{k=1}^{n-r}\sum_{j=1}^{\frac{r}{2}-1}h(k+j).
\end{aligned}
$$

Case 1. Suppose $n-r<\frac{r}{2}-1$. Then,
$$
\begin{aligned}
& F_h(3,n) > F_h(r,n) \\
& \Leftrightarrow \quad \sum_{j=n-r+1}^{n-2}(n-j)h(j)
\\
& \qquad \quad > \sum_{j=2}^{n-r}h(j)+\sum_{j=n-r+1}^{\frac{r}{2}-1}rh(j)+\frac{r}{2}h\Big(\frac{r}{2}\Big)+ \sum_{j=\frac{r}{2}+1}^{n-\frac{r}{2}}h(j)+ 2\sum_{k=1}^{n-r}\sum_{j=1}^{\frac{r}{2}-1}h(k+j)
\\
& \Leftrightarrow \quad (n-r)h\Big(\frac{r}{2}\Big)+\sum_{j=\frac{r}{2}+1}^{n-\frac{r}{2}}(n-j-1)h(j)+\sum_{j=n-\frac{r}{2}+1}^{n-2}(n-j)h(j)
\\
& \qquad \quad > \sum_{j=2}^{n-r}h(j)+\sum_{j=n-r+1}^{\frac{r}{2}-1}(r-n+j)h(j)+2\sum_{k=1}^{n-r}\sum_{j=1}^{\frac{r}{2}-1}h(k+j).
\end{aligned}
$$

Since $n-r<\frac{r}{2}-1$,
 \[
\begin{aligned}
2\sum_{k=1}^{n-r}\sum_{j=1}^{\frac{r}{2}-1}h(k+j)
& =
2h(2)+4h(3)+\cdots+2(n-r)h(n-r+1)+\cdots + 2(n-r)h\Big(\frac{r}{2}\Big)
\\
& \quad +2(n-r-1)h\Big(\frac{r}{2}+1\Big)+\cdots + 2h\Big(n-\frac{r}{2}-1\Big).
\end{aligned}
\]
Let
\[M=(n-r-1)+\frac12\Big(\frac{3r}{2}-n-1\Big)\Big(\frac{3r}{2}-n\Big)+(n-r)(2r-n-1).\]
Then, since $h$ is strictly increasing, it follows that
$$
\begin{aligned}
\sum_{j=2}^{n-r}h(j) \, + & \!\sum_{j=n-r+1}^{\frac{r}{2}-1}(r-n+j)h(j)+2\sum_{k=1}^{n-r}\sum_{j=1}^{\frac{r}{2}-1}h(k+j)
\\
& < M h\Big(\frac{r}{2}\Big)+2(n-r-1)h\Big(\frac{r}{2}+1\Big)+\cdots + 2h\Big(n-\frac{r}{2}-1\Big)
\\
& = M h\Big(\frac{r}{2}\Big)+\sum_{j=\frac{r}{2}+1}^{n-\frac{r}{2}-1}2\Big(n-\frac{r}{2}-j\Big)h(j).
\end{aligned}
$$
Therefore, since in
\[(n-r)h\Big(\frac{r}{2}\Big)+\sum_{j=\frac{r}{2}+1}^{n-\frac{r}{2}}(n-j-1)h(j)+\sum_{j=n-\frac{r}{2}+1}^{n-2}(n-j)h(j)\]
the function $h$ is evaluated always in numbers greater or equal than $\frac{r}{2}$
and Remark \ref{r:pares} gives
$$
(n-r)+\sum_{j=\frac{r}{2}+1}^{n-\frac{r}{2}}(n-j-1) +\sum_{j=n-\frac{r}{2}+1}^{n-2}(n-j)
= \sum_{j=2}^{n-r} 1 +\sum_{j=n-r+1}^{\frac{r}{2}-1}(r-n+j)+2\sum_{k=1}^{n-r}\sum_{j=1}^{\frac{r}{2}-1} 1,
$$
it suffices to check that
\[\sum_{j=\frac{r}{2}+1}^{n-\frac{r}{2}-1}2\Big(n-\frac{r}{2}-j\Big)h(j)<\sum_{j=\frac{r}{2}+1}^{n-\frac{r}{2}-1}(n-j-1)h(j).\]
Hence, it suffices to show for every $\frac{r}{2}+1 \le j \le n-\frac{r}{2}-1$ that
$$
2n-r-2j<n-j-1,
$$
and this is equivalent to $n-r<j-1$, which follows from
$$
j-1\geq \frac{r}{2}> \frac{r}{2}-1 > n-r.
$$
This finishes the proof of Case 1.

\smallskip

Case 2. Suppose $\frac{r}{2}-1= n-r$. Then,
$$
\begin{aligned}
& F_h(3,n)>F_h(r,n)
\\
& \Leftrightarrow \quad \sum_{j=2}^{\frac{r}{2}-1}\Big(\frac{3r}{2}-1-j\Big)h(j)+\sum_{j=\frac{r}{2}}^{\frac{3r}{2}-3}\Big(\frac{3r}{2}-1-j\Big)h(j)
\\
& \qquad \quad >  \sum_{j=2}^{\frac{r}{2}-1}\Big(\frac{3r}{2}-j\Big)h(j)+\frac{r}{2}h\Big(\frac{r}{2}\Big)+\sum_{j=\frac{r}{2}+1}^{r-1}h(j)+2\sum_{k=1}^{\frac{r}{2}-1}\sum_{j=1}^{\frac{r}{2}-1}h(k+j)
\\
& \Leftrightarrow \quad \sum_{j=\frac{r}{2}}^{\frac{3r}{2}-3}\Big(\frac{3r}{2}-1-j\Big)h(j)
\\
& \qquad \quad >  \sum_{j=2}^{\frac{r}{2}-1}h(j)+\frac{r}{2}h\Big(\frac{r}{2}\Big)+\sum_{j=\frac{r}{2}+1}^{r-1}h(j)+
2\sum_{k=1}^{\frac{r}{2}-1}\sum_{j=1}^{\frac{r}{2}-1}h(k+j).
\end{aligned}
$$

Notice that
\[
\begin{aligned}
2\sum_{k=1}^{\frac{r}{2}-1}\sum_{j=1}^{\frac{r}{2}-1}h(k+j) & =
2h(2)+4h(3)+\cdots+2\Big(\frac{r}{2}-1\Big)h\Big(\frac{r}{2}\Big)+ 2\Big(\frac{r}{2}-2\Big)h\Big(\frac{r}{2}+1\Big)+
\cdots  + 2h(r-2)
\\
& = \sum_{j=2}^{\frac{r}{2}}2(j-1)h(j)+\sum_{j=\frac{r}{2}+1}^{r-2}2(r-1-j)h(j).
\end{aligned}
\]

Since $h$ is strictly increasing, it follows that

\[
\begin{aligned}
& \sum_{j=2}^{\frac{r}{2}-1}h(j)+\frac{r}{2}h\Big(\frac{r}{2}\Big)+\sum_{j=\frac{r}{2}+1}^{r-1}h(j)+
2\sum_{k=1}^{\frac{r}{2}-1}\sum_{j=1}^{\frac{r}{2}-1}h(k+j)
\\
& \qquad < \Big(\frac{r}{2}-1\Big)\Big(\frac{r}{2}+2\Big)h\Big(\frac{r}{2}\Big)+
\sum_{j=\frac{r}{2}+1}^{r-1}h(j) +\sum_{j=\frac{r}{2}+1}^{r-2}2(r-1-j)h(j).
\end{aligned}\]

Thus, it suffices to check that for every $\frac{r}{2}+1\leq j$ we have $2r-1-2j<\frac{3r}{2}-1-j$, which is direct.

\smallskip

Case 3. Suppose $\frac{r}{2}-1< n-r$. Then,
$$
\begin{aligned}
& F_h(3,n)>F_h(r,n)
\\
& \Leftrightarrow \quad \sum_{j=2}^{n-r}(r-1)h(j)+\sum_{j=n-r+1}^{n-2}(n-j)h(j)
\\
& \qquad \quad > \sum_{j=2}^{\frac{r}{2}-1}rh(j)+\frac{r}{2}h\Big(\frac{r}{2}\Big)+\sum_{j=\frac{r}{2}+1}^{n-\frac{r}{2}}h(j)+2\sum_{k=1}^{n-r}\sum_{j=1}^{\frac{r}{2}-1}h(k+j)
\\
& \Leftrightarrow \quad
\Big(\frac{r}{2}-1\Big)h\Big(\frac{r}{2}\Big)+\sum_{j=\frac{r}{2}+1}^{n-r}(r-2)h(j)+\sum_{j=n-r+1}^{n-2}(n-j)h(j)
\\
& \qquad \quad > \sum_{j=2}^{\frac{r}{2}-1}h(j)+\sum_{j=n-r+1}^{n-\frac{r}{2}}h(j)+2\sum_{k=1}^{n-r}\sum_{j=1}^{\frac{r}{2}-1}h(k+j).
\end{aligned}
$$

Since $\frac{r}{2}-1< n-r$,
\[
\begin{aligned}
2\sum_{k=1}^{n-r}\sum_{j=1}^{\frac{r}{2}-1}h(k+j)
& =
2h(2)+4h(3)+\cdots+2\Big(\frac{r}{2}-1\Big)h\Big(\frac{r}{2}\Big)+\cdots + 2\Big(\frac{r}{2}-1\Big)h(n-r+1)
\\
& \qquad +2\Big(\frac{r}{2}-2\Big)h(n-r+2)+\cdots + 2h\Big(n-\frac{r}{2}-1\Big)
\\
& = \sum_{j=2}^{\frac{r}{2}}2(j-1)h(j)+\sum_{j=\frac{r}{2}+1}^{n-r}(r-2)h(j)+\sum_{j=n-r+1}^{n-\frac{r}{2}-1}(2n-r-2j)h(j).
\end{aligned}
\]

Since $h$ is strictly increasing, it follows that
\[
\begin{aligned}
& \sum_{j=2}^{\frac{r}{2}-1}h(j)+\sum_{j=n-r+1}^{n-\frac{r}{2}}h(j)+2\sum_{k=1}^{n-r}\sum_{j=1}^{\frac{r}{2}-1}h(k+j)
\\
& <\Big(\frac{r^2}{4}-2\Big)h\Big(\frac{r}{2}\Big)+\sum_{j=n-r+1}^{n-\frac{r}{2}}h(j)+\sum_{j=\frac{r}{2}+1}^{n-r}(r-2)h(j)+\sum_{j=n-r+1}^{n-\frac{r}{2}-1}(2n-r-2j)h(j).
\end{aligned}
\]

Therefore, since in
\[  \Big(\frac{r}{2}-1\Big)h\Big(\frac{r}{2}\Big)+\sum_{j=\frac{r}{2}+1}^{n-r}(r-2)h(j)+\sum_{j=n-r+1}^{n-2}(n-j)h(j)\]
the function $h$ is evaluated always in numbers greater or equal than $\frac{r}{2}$, 
it suffices to check that  for every $n-r+1\leq j $ we have $2n-r-2j+1\leq n-j$, which is immediate.
\end{proof}

\begin{lemma}\label{l:max_odd} Given $5\leq r \leq n$ with $r$ odd and a strictly increasing function $h:\ZZ^+ \rightarrow \RR$, then $F_h(3,n)>F_h(r,n)$. If $h$ is a strictly decreasing function, then $F_h(3,n)<F_h(r,n)$.
\end{lemma}

\begin{proof}
The second statement is a consequence of the first one, if we consider the function $-h$.
Therefore, we can assume that $h$ is a strictly increasing function.

Let us recall that $$F_h(3,n)=nh(1)+\sum_{j=2}^{n-2}(n-j)h(j).$$

Case 1. Suppose $n-r < \frac{r-1}{2}$. Then,
$$
\begin{aligned}
F_h(r,n)= & \sum_{j=1}^{n-r}(n+1-j)h(j)+\sum_{j=n-r+1}^{\frac{r-1}{2}}rh(j)+2\sum_{k=1}^{n-r}\sum_{j=1}^{\frac{r-1}{2}}h(k+j).
\end{aligned}
$$
Therefore,
$$
\begin{aligned}
& F_h(3,n)>F_h(r,n) \\
& \Leftrightarrow \quad \sum_{j=n-r+1}^{n-2}(n-j)h(j) >  \sum_{j=2}^{n-r}h(j)+\sum_{j=n-r+1}^{\frac{r-1}{2}}rh(j)+2\sum_{k=1}^{n-r}\sum_{j=1}^{\frac{r-1}{2}}h(k+j)
\\
& \Leftrightarrow \quad \sum_{j=\frac{r+1}{2}}^{n-2}(n-j)h(j) > \sum_{j=2}^{n-r}h(j)+
\sum_{j=n-r+1}^{\frac{r-1}{2}}(j-n+r)h(j)+2\sum_{k=1}^{n-r}\sum_{j=1}^{\frac{r-1}{2}}h(k+j).
\end{aligned}
$$

Since $n-r<\frac{r-1}{2}$,
\[
\begin{aligned}
2\sum_{k=1}^{n-r}\sum_{j=1}^{\frac{r-1}{2}}h(k+j) & =
2h(2)+\cdots+2(n-r)h(n-r+1)+\cdots
\\
& \quad + 2(n-r)h\Big(\frac{r+1}{2}\Big)+\cdots + 2h\Big(n-\frac{r}{2}-\frac{1}{2}\Big).
\end{aligned}
\]
Let
$$N=n-r-1+\frac12 \frac{(3r-2n-1)}{2}\frac{(3r-2n+1)}{2}+(n-r)(n-r+1)+(3r-2n-1)(n-r).$$
Therefore, since $h$ is strictly increasing,
\[
\begin{aligned}
& \sum_{j=2}^{n-r}h(j)+ \sum_{j=n-r+1}^{\frac{r-1}{2}}(j-n+r)h(j)+2\sum_{k=1}^{n-r}\sum_{j=1}^{\frac{r-1}{2}}h(k+j)
\\
& \qquad < N h\Big(\frac{r+1}{2}\Big)+\sum_{j=\frac{r+3}{2}}^{n-\frac{r}{2}-\frac12}2\Big(n-\frac{r}{2}+\frac12-j\Big)h(j).
\end{aligned}
\]
Thus, it suffices to check that for every $\frac{r+3}{2}\leq j \leq n-\frac{r}{2}-\frac12$ we have
\[2n-r+1-2j<n-j,\]
and this is equivalent to $n-r<j-1$,
which is immediate since \[j-1\geq \frac{r+1}{2}>n-r.\]

Case 2. Suppose $\frac{r-1}{2}\leq n-r$. Then,
$$
\begin{aligned}
F_h(r,n) & = \sum_{j=1}^{\frac{r-1}{2}}rh(j)+\sum_{j=1}^{\frac{r-1}{2}}(n-r+1-j)h(j)+\sum_{j=\frac{r+1}{2}}^{n-r}(n-r+1-j)h(j)+2\sum_{k=1}^{n-r}\sum_{j=1}^{\frac{r-1}{2}}h(k+j)
\\
& = \sum_{j=1}^{\frac{r-1}{2}} (n+1-j)h(j)+\sum_{j=\frac{r+1}{2}}^{n-r}(n-r+1-j)h(j)+2\sum_{k=1}^{n-r}\sum_{j=1}^{\frac{r-1}{2}}h(k+j).
\end{aligned}
$$
Therefore,
$$
\begin{aligned}
& F_h(3,n)>F_h(r,n)
\\
& \Leftrightarrow \quad \sum_{j=\frac{r+1}{2}}^{n-2}(n-j)h(j) > \sum_{j=2}^{\frac{r-1}{2}}h(j)+\sum_{j=\frac{r+1}{2}}^{n-r}(n-r+1-j)h(j)+2\sum_{k=1}^{n-r}\sum_{j=1}^{\frac{r-1}{2}}h(k+j)
\\
& \Leftrightarrow \quad \sum_{j=\frac{r+1}{2}}^{n-r}(r-1)h(j) +\sum_{j=n-r+1}^{n-2}(n-j)h(j)> \sum_{j=2}^{\frac{r-1}{2}}h(j)+ 2\sum_{k=1}^{n-r}\sum_{j=1}^{\frac{r-1}{2}}h(k+j).
\end{aligned}
$$

Since $\frac{r-1}{2}\leq n-r$,
\[
\begin{aligned}
2\sum_{k=1}^{n-r}\sum_{j=1}^{\frac{r-1}{2}}h(k+j)
& =2h(2)+\cdots+2\Big(\frac{r-1}{2}\Big)h\Big(\frac{r+1}{2}\Big)+\cdots
\\
& \quad + 2\Big(\frac{r-1}{2}\Big)h(n-r+1)+\cdots +2h\Big(n-\frac{r}{2}-\frac12\Big).
\end{aligned}
\]

Let $N'=\frac{r^2+2r-7}{4}$. Therefore, since $h$ is strictly increasing,
\[
\sum_{j=2}^{\frac{r-1}{2}}h(j)+ 2\sum_{k=1}^{n-r}\sum_{j=1}^{\frac{r-1}{2}}h(k+j)
< N'h\Big(\frac{r+1}{2}\Big)+\sum_{j=\frac{r+3}{2}}^{n-r}(r-2)h(j)+\sum_{j=n-r+1}^{n-\frac{r}{2}-\frac12}(2n-r+1-2j)h(j).
\]
Thus, it suffices to check that for every $\frac{r+3}{2}\leq j \leq n-r$ we have $r-2\leq r-1$, and for every $n-r+1\leq j \leq n-\frac{r}{2}-\frac12$ we have
\[
2n-r+1-2j\leq n-j \quad \Leftrightarrow \quad n-r+1\leq j,
\]
which is trivial.
\end{proof}

\begin{example} Notice that given $3< r\leq n$ and a strictly increasing function $h:\ZZ^+ \rightarrow \RR$, although $F_h(3,n)>F_h(r,n)$, $F_h(r,n)$ is not necessarily decreasing on $r$. For example, suppose $h$ is the identity map. Then $F_h(12,13)=13h(1)+14(h(2)+h(3)+h(4)+h(5))+8h(6)+h(7)$ and $F_h(11,13)=13h(1)+14h(2)+15(h(3)+h(4) +h(5))+4h(6)+2h(7)$. Thus, $F_h(12,13)-F_h(11,13)=4h(6)-h(3)-h(4)-h(5)-h(7)=24-3-4-5-7=5>0$
and $F_h(12,13)>F_h(11,13)$.
\end{example}

Let us recall the following definitions from \cite{ERY}.
A vertex of degree at least
three in a graph $G$ will be called a \emph{major} vertex of $G$.
Any end-vertex (a vertex of degree one) $u$ of $G$ is said to be a
\emph{terminal} vertex of a major vertex $v$ of $G$ if
$d_G(u,v) < d_G(u,w)$ for every other major vertex $w$ of $G$.
The \emph{terminal degree} of a major vertex $v$ is the
number of terminal vertices of $v$. A major vertex $v$ of $G$ is
called \emph{exterior} major vertex if it has positive terminal
degree.
Let $\mathcal{M}(G)$ be the set of exterior major vertices of
$G$ having terminal degree greater than one.

\smallskip

The following result is elementary.

\begin{lemma} \label{l:major}
If $T$ is a tree, then $\mathcal{M}(T)=\emptyset$ if and only if $T$ is a path graph.
\end{lemma}


Given a unicyclic graph $G$ and a vertex $v$
let us denote $G_v$ the connected component of $G\setminus v$ intersecting the cycle. Then, the complement $T_v:=G\setminus G_v$ is a tree (possibly being the the single vertex $v$).

\begin{theorem} Let $G$ be a unicyclic graph with $n\geq 6$ vertices.
\begin{itemize}
	\item[$(1)$] If $h$ is a strictly increasing function, then
\[
n\,h(1)+\frac{1}{2} n(n-3)\, h(2) \leq W_h(G)\leq F_h(3,n),
\]
the lower bound is attained if and only if $G=J_n$ and the upper bound is attained if and only if $G=G_{3,n}$.

	\item[$(2)$] If $h$ is a strictly decreasing function, then
\[
F_h(3,n) \leq W_h(G)\leq n\,h(1)+\frac{1}{2} n(n-3)\, h(2),
\]
the lower bound is attained if and only if $G=G_{3,n}$ and the upper bound is attained if and only if $G=J_n$.
\end{itemize}
\end{theorem}

\begin{proof}
The second statement is a consequence of the first one, if we consider the function $-h$.
Hence, we can assume that $h$ is a strictly increasing function.

Since $G$ is a unicyclic graph with $n$ vertices, there are $n$ edges and $n$ pairs of adjacent vertices.
Thus, there are $\binom{n}{2}-n=\frac12 n(n-3)$ pairs of vertices at distance at least 2.
The lower bound is attained if and only if $T$ has diameter 2.
Notice that this means that the cycle has at most 5 vertices and since $n\geq 6$ and any vertex which is not in the cycle is at distance at most 2 from any vertex in it, the cycle is necessarily $C_3$.
Therefore, the lower bound is attained if and only if $G=J_n$.

Assume that $G$ is a unicyclic graph with $n$ vertices such that $W_h(G)$ is maximal and let $C$ be the cycle.

Seeking for a contradiction, consider any vertex $v\in G$ such that $v$ is not in the unique cycle and $deg(v)\geq 3$ or $v$ is in the cycle and $deg(v)\geq 4$. Then,  there is an exterior major vertex $w$ in $T_v$ (where possibly $v=w$). Hence, consider two terminal vertices of $w$, $u_1,u_2$ with $d(w,u_1)=k$. Let us define new vertices $\{v_1,\dots, v_k\}$ and let
$$
G':=\big(G\setminus [w,u_1]\big)\cup u_2v_1\cup v_1v_2\cup \cdots \cup v_{k-1}v_k .
$$
Thus, it is immediate to see that, since $h$ is strictly increasing, $W_h(G')>W_h(G)$ leading to contradiction.
Therefore, $deg(v)\leq 2$ for every vertex $v\notin C$, and $deg(v)\leq 3$ for every vertex $v\in C$.

Suppose there exist $v_1,v_2\in C$ with $deg(v_i)=3$ for $i=1,2$. Then, $T_{v_i}$ is a path $[v_i,u_i]$ with length $d([v_i,u_i])=l_i$. Let $D_i:=\sum_{v\in G\setminus \{T_{v_1}\cup T_{v_2}\}}d(v_i,v)$ for $i=1,2$. If $D_i<D_j$ and $l_i>l_j$, consider the vertex $w\in [v_i,u_i]$ such that $d(v_i,w)=l_j$, let $s=l_i-l_j$ and let us define new vertices $\{w_1,\dots,w_{s}\}$. Then, let
$$
G':=\big(G\setminus [w,u_i]\big)\cup u_jw_1\cup w_1w_2\cup \cdots \cup w_{s-1}w_{s} .
$$
Since $D_i<D_j$ and $h$ is strictly increasing, it follows that $W_h(G')>W_h(G)$ leading to contradiction. Thus, we may assume, relabeling if necessary, that $D_1\leq D_2$ and $l_1\leq l_2$.
Hence, let us define new vertices $\{a_1,\dots,a_{l_1}\}$ and let
$$
G':=\big(G\setminus [v_1,u_1]\big)\cup u_2a_1\cup a_1a_2\cup \cdots \cup a_{l_1-1}a_{l_1} .
$$
Since $D_1\leq D_2$, $l_1\leq l_2$ and $h$ is strictly increasing, it follows that $W_h(G')>W_h(G)$ leading to contradiction.

Therefore, $G$ is a unicyclic graph with at most one vertex with degree 3 and $G=G_{r,n}$ where $r$ is the length of the cycle. By lemmas \ref{l:max_even} and \ref{l:max_odd}, $G=G_{3,n}$.
\end{proof}

\begin{corollary} If $G$ is a unicyclic graph with $n\geq 6$ vertices, then
\[
n(n-2) \leq W(G)\leq 1+\sum_{k=1}^{n-2} (n-k)\,k=\frac{n^3-7n+12}{6},
\]
the lower bound is attained if and only if $G=J_n$ and the upper bound is attained if and only if $G=G_{3,n}$.
\end{corollary}

\end{document}